\documentclass[10pt,twoside,leqno]{amsart}

\usepackage{amssymb}
\theoremstyle{plain}
\newtheorem{thm}{Theorem}[section]

\newtheorem{lem}[thm]{Lemma}

\theoremstyle{definition}

\theoremstyle{remark}

\numberwithin{equation}{section}

\begin{document}
\setcounter{page}{1}


\vspace*{1.5cm}

\title[  ]
{\large A mathematical proof for the existence of a possible source for dark energy }
\author[I.Mohamadi]{Issa Mohamadi}
\address { \textit{Department of Mathematics, Azad University - Sanandaj Branch,
  Sanandaj, Iran}}
  
\address { \textit{ Tel: +1 - 647 405 3379, +98 - 918 171 9658 }}
\address {\textit{ E-mail address:
  imohamadi@iausdj.ac.ir\\ imohamadi.maths@gmail.com}}

\date{}
\maketitle

\vspace*{-0.5cm}

\bigskip
\begin{abstract}The expansion of the universe has been accepted by scientists for more than a century. However, since the 1990s, observations have suggested that the universe is expanding at an accelerating rate \cite{Perlmutter, Riess}. Although the source of this acceleration is still unknown, cosmologists call it dark energy; a type of mysterious energy which exerts a negative pressure on the universe. In this paper, applying different approach from those of classical mathematical physics based mainly on functional analysis, we propose an answer to solve this mystery. In fact, we first establish a theorem for multivalued mappings; and as a result we apply this theorem to show the existence of a remnant field of high energy in a continuously expanding system of energy. The results may also be applied to the early universe to show the existence of a possible source for dark energy. Our answer is also  confirmed partially by a recent paper published in Nature Astronomy \cite{Zhao}. \end{abstract}
\maketitle
\bigskip



\section{Introduction }
 Let $X$ and $Y$ be two topological vector spaces, we recall that a multivalued mapping $T:X\rightarrow 2^{Y}$ is said to be upper semicontinuous, if for each open subset $V$ of $Y$ and each $x\in X$  with $T(x)\subseteq V$, there exists an open neighborhood $U$ of $x$ in $X$ such that $T(y)\in V$ for all $y\in U$.
 We say that a mapping $T$ from $X$ into $2^X$ is convex if $\lambda t+(1-\lambda)z\in T(\lambda x+(1-\lambda)y)$, for all $t\in T(x)$, $z\in T(y)$ and $\lambda\in (0,1)$. We also recall that for a multivalued mapping $T$ from $X$ into $2^{X}$, $x\in X$ is a fixed point of $T$ if $x\in T(x)$.  For more details on these and related concepts refer to \cite{Aubin}.

  There are a number of landmark fixed point theorems for multivalued mappings. In $1941$, Kakutani showed that if $C$ is a nonempty convex compact subset of an $n$-dimensional Euclidean space $\mathbb{R}^{n}$ and $T$ from $C$ into $2^{C}$ is an upper semicontinuous mapping such that $T(x)$ is a nonempty convex closed subset of $C$ for all $x\in C$; then, $T$ possesses a fixed point in $C$  \cite{Kakutani}. Also In $1951$, Glicksberg  and in $1952$, Ky-Fan, independently, generalized Kakutani's fixed point theorem from Euclidean spaces to locally convex vector spaces \cite{Fan, Glicksberg}. The Kakutani-Glicksberg-Fan theorem is the main tool to obtain our theorem. An application of our multivalued fixed point theorem is to prove the existence of a high energy field of a continuously expanding system of energy after being expanded by a consecutive countable  number of multivalued mappings.\par  \section{Our results}

The following Lemma will be used to prove our fixed point theorem.
\begin{lem} Let $C$ be a nonempty convex compact subset of a locally convex Hausdorff vector space $X$, and $\{U_{i}\}_{i\in I}$ be a descending chain of nonempty closed subsets of $C$. Assume also that $T:C\rightarrow2^{C}$ is an upper semicontinuous mapping
 such that $T(x)$ is a nonempty closed convex subset of $X$ for all $x\in C$. Then, $T(\bigcap_{i\in I} U_{i})=\bigcap_{i\in I} T(U_{i})$.     \end{lem}
\begin{proof} It is obvious that $T(\bigcap_{i\in I} U_{i})\subseteq\bigcap_{i\in I} T(U_{i})$. We show the reverse inclusion. Suppose that $w$ is an arbitrary element in $\bigcap_{i\in I} T(U_{i})$. Then, for each $i\in I$ there exists $z_{i} \in U_{i}$ such that $w\in T(z_{i})$. Now, we define the relation $\leq$ on $I$ as : $i\leq j$ iff $U_{j} \subseteq U_{i}$, for all $i,j\in I$. Then, $(I,\leq)$ is a directed set since it is a totally ordered set. Thus, for an arbitrary $i_{0}\in I$, $\{z_{i}\}_{i\geq i_{0}}$ is a net in $U_{i_{0}}$. Therefore, there exists a subnet $\{z_{\alpha_{i}}\}_{i\in J}$ that converges to $z$ in $U_{i_{0}}$, as $U_{i_{0}}$ is a compact set. It is clear that $z\in \bigcap_{i\in I} U_{i}$. We prove that $w\in T(z)$. Suppose, on contrary, that $w$ doesn't belong to $T(z)$. Since $T(z)$ is a nonempty compact convex subset of $X$, by Hahn-Banach separation theorem, it yields that there exists $\varphi \in X^{*}$, the dual space of $X$, and $\lambda \in \mathbb{R^{+}}$ such that $\varphi(w)<\lambda<\varphi(y)$, for all $y\in T(z)$. Let $V$ be any open set containing $T(z)$ such that $\lambda <\varphi(y)$, for all $y\in V$. It follows , by upper semicontinuity of $T$, that there exists a neighborhood $N(z)$ of $z$ such that for all $p\in N(z)$, including some $z_{\alpha}$, we have $T(p)\subseteq V$. That is, $\lambda<\varphi(y)$ for all $y\in T(p)$. This contradicts the fact that $w\in T(z_{i})$, for all $i\in I$. Accordingly, $w\in T(z)$. This completes the proof.\end{proof}

The following example shows that the underlying subsets of $C$ in Lemma $(2.1)$ necessarily need to be compact.

{\bf Example.} Let $X=\mathbb{R}$ and $C=[0,1]$. Define $T$ from $X$ into $2^{X}$ as $T(x)=[x,1]$ for $x\in C$, and $T(x)=\{x\}$ for $x\in X-C$. It is easy to verify that $T:C\rightarrow 2^{C}$ is an upper semicontinuous mapping. Also, define $U_{n}=(0,\frac{1}{n}]$ for $n\in \mathbb{N}$. Then, it is clear that $T(\bigcap_{n=1}^\infty U_{n})=\emptyset$ since $\bigcap_{n=1}^\infty U_{n}=\emptyset$. However, $\bigcap_{n=1}^\infty T( U_{n})=(0,1]$, as $T(U_{n})=(0,1]$, for all $n\in \mathbb{N}$.\\

In the following, we establish a fixed point theorem for convex multivalued mappings which is acting a great role in showing the existence of a high energy field in an expanding system of energy. 
\begin{thm} Let $X$ be a locally convex Hausdorff vector space, and $C$ be a nonempty convex compact subset of $X$. Assume that $T:C\rightarrow 2^{X}$ is a multi-valued convex upper semicontinuoues mapping such that $T(x)$ is a compact subset of $X$ for all $x\in C$. If $C\subseteq T(C)$, then there exists $x_{0}\in C$ such that $x_{0}\in T(x_{0})$.  \end{thm}

\begin{proof}
Let $$\Gamma=\{U\subseteq C: U is\:nonempty,\:closed,\:convex\:and\:U\subseteq T(U)\}.$$
Then, $(\Gamma,\subseteq)$, where $\subseteq$ is inclusion, is a partially ordered set. Also, by Lemma $2.1$, every descending chain in $\Gamma$ has a lower bound in $\Gamma$. Therefore, by Zorn's lemma, $\Gamma$ has a minimal element, say $U_{0}$. We show that $U_{0}$ is singleton. Define $S:U_{0}\rightarrow 2^{U_{0}}$ by $S(x)=T(x)\cap U_{0}$ for all $x\in U_{0}$. Then, $S(x)$ is a convex, compact subset of $X$, for all $x\in U_{0}$, since $T(x)$ and $U_{0}$ are convex and compact. We prove that $S(x)$ is nonempty. Suppose, on contrary, that there exists $x\in U_{0}$ such that $S(x)$ is empty. Let $V=\{y\in U_{0} : S(x)\:\: is\:\:nonempty\}$. Then, $V$ is nonempty as $U_{0}\subseteq T(U_{0})$. We also have $V\subseteq T(V)$; otherwise, there exists $y\in V$ such that $y$ does not belong to $T(V)$. Thus, $y\in U_{0}\subseteq T(U_{0})$. That is, there exists $z\in U_{0}-V$ such that $y\in T(z)$. On the other hand, according to the definition of $V$, $T(z)\cap U_{0}=\emptyset$. This contradicts the fact that $y\in T(z)\cap V\subseteq T(z)\cap U_{0}$. By convexity and upper semicontinuity of $T$, it can easily be seen that $V$ is a nonempty, convex and compact subset of $U_{0}$ so that $V\subseteq T(V)$ and $V\neq U_{0}$. This contradicts the minimality of $U_{0}$. Accordingly, $S(x)$ is nonempty for all $x\in U_{0}$. Hence, by Ky-Fan's fixed point theorem, there exists $x_{0}\in U_{0}$ so that $x_{0}\in T(x_{0})$. Therefore, $\{x_{0}\}\subseteq T\{x_{0}\}$. Finally, we have $U_{0}=\{x_{0}\}$. This completes the proof.
\end{proof}

{\bf Example.} As a natural example of theorem $2.2$, consider a homogenous elastic cube pulled by same forces along its diagonals; then its center remains fixed. \\

{\bf Remark.} In theorem $2.2$, $T(x)$ does not necessarily  need to be nonempty for all $x\in C$. Indeed, the set $D=\{x\in C: T(x)\:is\: nonempty\}$ is a convex set since $T$ is a convex multivalued mapping. Also, by upper semicontinuity of $T$, it is easy to show that $D$ is a closed and therefore a compact subset of $X$. It is also obvious that $D\subseteq T(D)$. Accordingly, $T(x)$ does not need to be nonempty for all $x\in X$. \\
  
In what follows every single point $x$ in an expanding system of energy $C$ in $\mathbb{R}^{m}$  under the multivalued mapping $T:C \longrightarrow 2^{\mathbb{R}^{m}}, m\in\mathbb{N}$, takes $T(x)$. For experimental purpose, we may consider $x$ as a tiny compact convex set of measure almost zero. We also denote by $\delta(x)$, $V(T(x))$ and $d(x,T(x))$,  respectively, the density of $x$, the volume of $T(x)$ and the distance of $x$ from $T(x)$, defined by

$$d(x,T(x))=\inf \{\Vert x-y\Vert : y\in T(x)\}.$$

\begin{thm}
Let $C$ be an a nonempty compact convex  expanding set of energy (e.g. a ball or a cube) in $\mathbb{R}^{m}, m\in\mathbb{N}$. Suppose that $\{t_{n}\}$ is an increasing sequence of time and for each $n$, $T_{t_{n}}:C_{t_{n-1}}\longrightarrow 2^{\mathbb{R}^{m}}$  is a convex and compact upper semicontinuous multivalued mapping that satisfies
 $$C_{t_{n-1}}\subset T_{t_{n}}(C_{t_{n-1}}),$$ where, $C_{t_{n}}=T_{t_{n}}(C_{t_{n-1}})$, and $C_{t_{0}}=C$. If the the whole system satisfies the following conditions:\par 
 
 $(1)$ \:\: $\delta (y)=f_{n}(\delta(x), V(T_{t_{n}}(x)),d(x,T_{t_{n}}(x)))$, for all $x\in C_{t_{n-1}}$ and $y\in T_{t_{n}}(x)$; where $f_{n}$ is a triple variable and nonnegative valued function which decreases as either of $V(T_{t_{n}}(x))$ and $d(x,T_{t_{n}}(x))$ increases, (it may also be presumable to suppose that 
$V(T_{t_{n}}(x))$ increases if $d(x,T_{t_{n}}(x))$ does, as the system is expanding).\par
$(2)$For $x\neq y$, $T(x)\cap T(y)=\emptyset$.\par 
Then there exists a field of high energy, relative to other points, in $C_{t_{n}}$, for all $n=1,2,...$.      
\end{thm}
\begin{proof}
Let 
$$F_{n}=\bigcap_{i=1}^{n}Fix(T_{t_{i}}).$$
We shall prove $F_{n}$ is nonempty. The proof is by induction. For $n=1$, since $C_{0}\subset T_{t_{1}}(C_{0})$ and all the required conditions hold, by Theorem $2.2$, $F_{1}$ is nonempty  and also convex as $T_{t_{1}}$ is a convex multivalued mapping. Now suppose that $F_{n-1}$ is nonempty; then for all $x\in F_{n-1}$, by definition, $d(x,T_{t_{i}}(x))=0$ for all $i=1,...,n-1$. Accordingly, by condition $(1)$, $F_{n-1}$ is denser than $C_{n-1}\setminus F_{n-1}$. On the other hand, if  $x\in T_{t_{n}}(y)$, for some $x\in F_{n-1}$ and $y\in C_{n-1}\setminus F_{n-1}$; then we have $\delta(y)>\delta(x)$, a contradiction, cosidering condition $(1)$. Therefore,
$$F_{n-1} \subset T_{t_{n}}(F_{n-1}).$$
Accordingly, since $F_{n-1}$ is convex closed, and therefore compact, subset of $C_{t_{n-1}}$, by Theorem $2.2$ again, it follows that $T_{t_{n}}$ possesses a fixed point in $F_{n-1}$. That is, $F_{n}$ is a nonempty convex compact subset of $\mathbb{R}^{n}$. Hence, $\{F_{n}\}$ is a descending sequence of nonempty convex compact subsets of $C$; thus their intersection is nonempty. This implies by conditions $(1)$  that 

$$F=\bigcap_{n=1}^{\infty}F_{n}=\bigcap_{n=1}^{\infty}Fix(T(t_{n}))$$
is a nonempty convex compact with higher density subset of the final system.   

\end{proof}

{\bf Remark.} If the early universe, shortly after the big bang, as an expanding system of energy satisfies the conditions mentioned in Theorem $2.3$, then there should exist a dynamical energy field with very high density emitting energy (and possibly matter) into the universe now and it is blowing up the space to expand it. This field exerts a force on the universe and can be a source for dark energy and dark matter, and even for cosmic microwave background (CMB). Accordingly, this remnant field may work like the early universe.\par  
{\bf Remark.} In Theorem $2.3$ it is enough to impose the conditions only on $Fix(T_{n})$; that is, we may assume that an expanding high enough energy field needs to satisfy the conditions.

\end{document}